\documentclass[12pt,dvipdfmx]{scrartcl}

\usepackage{array}
\usepackage{supertabular}
\usepackage{amsmath,amsfonts,amssymb,amsthm} 
\usepackage{bm}
\usepackage{graphicx}
\usepackage{subcaption}

\allowdisplaybreaks[4]

\newtheorem{lemma}{Lemma}
\newtheorem{theo}{Theorem}

\title {Asymptotic efficiency of M.L.E. using prior survey in multinomial distributions}
\author{Yo Sheena\thanks{Faculty of Economics, Shinshu University, Japan; Faculty of Data Science, Shiga University, Japan. }}

\date{April, 2019}

\begin{document}
\maketitle

\begin{abstract}
Incorporating information from a prior survey is generally supposed to decrease the estimation risk of the present survey. This paper aims to show how the risk changes by incorporating the information of a prior survey through watching the first and the second-order terms of the asymptotic expansion of the risk. We recognize that the prior information is of some help for risk reduction when we can acquire samples of a sufficient size for both surveys. Interestingly, when the sample size of the present survey is small, the use of the prior survey can increase the risk. In other words, blending information from both surveys can have a negative effect on the risk. Based on these observations, we give some suggestions on whether or not to use the results of the prior survey and the sample size to use in the surveys  for a reliable estimation.
\end{abstract}
\noindent
MSC(2010) \textit{Subject Classification}: Primary 62F10; Secondary 62F12
\section{Introduction}
Combining data has been an important topic in statistics for a long time. In 1950's,  some influential results on the topic like Cochran \cite{Cochran}, Cochran and Cox \cite{Cochran&Cox}, Good \cite{Good_1} were published. Later on, the scope and depth of  the research over the topic has developed and combining data becomes indispensable method in various fields.  Meta-analysis is now an important issue in many statistical fields not to mention its usefulness in medical research. In the field of survey sampling, the use of the auxiliary data, i.e. the combination of principle and other data, is a basic tool for efficient sampling (e.g. see Valliant et.al Chapter 14 of \cite{Valliant_et_all}). We also notice that combining data from existing multiple surveys has been a challenging task for econometrics (e.g. see Ridder and Moffitt \cite{Ridder&Moffitt}). 

In combining data, we have to pay attention to ``homogeneity'' between the data sets, which is defined in various ways according to the purpose of the research.  When the data sets are not homogeneous,  it is not easy to distill information unbiasedly from these data sets. Many elaborate statistical methods have been developed and also used in practice for this purpose. 

On the contrary, when the data sets are completely homogeneous, simple aggregation or pooling work well for condensing them. However it is noteworthy that though the estimation technique is quite simple, there still remains a fundamental question; How much more estimation efficiency is gained by using the information from another data set (survey) ? We try to answer this question in this paper under very simple setting; one sample from a multinomial distribution and the other sample from a multinomial distribution that is more coarse.

We consider the multinomial distribution model over the categories $C_i,\ i=1,\ldots,p+1,$ where the parameter (the probability for each category) is given by 
\begin{equation}
\label{para_multi_dist}
m \triangleq (m_1,\ldots,m_{p+1}).
\end{equation}
If there are no restrictions other than
\[
\sum_{i=1}^{p+1} m_i=1,\qquad m_i >0,\quad i=1,\ldots,p+1,
\]
we call the distribution model a ``full model'' (the dimension of the model equals $p$), while if some restrictions on the parameter reduce the model dimension to less than $p$, we call the model a ``submodel.''

Let $\widehat{m}\triangleq(\widehat{m}_1,\ldots,\widehat{m}_{p+1})$ be maximum likelihood estimator (MLE) of $m\triangleq (m_1,\ldots,m_{p+1})$. We evaluate the discrepancy between the predictive distribution given by MLE and the true distribution by using Kullback--Leibler divergence:
\[
\overset{-1}{D}[\widehat{m}:m]=\sum_{i=1}^{p+1} \widehat{m}_i \log \frac{\widehat{m}_i}{m_i}.
\]
Because Kullback--Leibler divergence is an ``$\alpha$-divergence'' with $\alpha=-1$, we will use the notation $\overset{-1}{D}$. (E.g., see Amari \cite{Amari4} for the $\alpha$-divergence.)
We evaluate the performance of MLE through its risk, that is, 
\begin{equation}
\label{def_ED^-1}
\overset{-1}{ED}=\overset{-1}{ED}[\widehat{m}: m]\triangleq E\bigl[\overset{-1}{D}[\widehat{m}: m]\bigr].
\end{equation}
It is difficult to obtain the explicit form of the risk. In this paper, alternatively, we derive the asymptotic expansion of the risk with respect to the sample size $n$ up to the second-order term. The first-order term ($n^{-1}$-order term) and the second-order term ($n^{-2}$-order term) give us important information on the asymptotic performance of MLE. 

For the full model, the asymptotic expansion with respect to the sample size $n$ is given by
\begin{equation}
\label{ED_-1_expan_full}
\overset{-1}{ED}=\frac{p}{2n}+\frac{1}{12n^2}(M-1)+o(n^{-2}),
\end{equation}
where 
\begin{equation}
M \triangleq \sum_{i=1}^{p+1} m^{-1}_i.
\end{equation}
(see (42) of  Sheena \cite{Sheena}).
The first-order term is determined by the ratio of the model's dimension to the sample size (``$p-n$ ratio''). Because this holds true for any parametric model, the first-order term for a submodel of $p'$-dimensions equals $p'/(2n)$. However, the second-order term of most submodels is too complicated to be derived explicitly. (See (25) of \cite{Sheena}. The asymptotic expansion of the risk with respect to $\alpha$-divergence for a general parametric distribution is given in Theorem 1 of  \cite{Sheena}.)

Suppose that random sampling is carried out from the multinomial distribution given by \eqref{para_multi_dist}. We will call this sampling the ``present'' survey. Additionally, suppose that there is another survey (sampling) related to the relative frequencies of the categories  $C_i$'s. We call this survey the ``prior'' survey. (Note that the words ``present'' and ``prior'' do not necessarily mean the chronological order.)

There are many possible kinds of relationship between the present and prior surveys. In this paper, we focus on the case when the prior survey is ``coarse" compared with the present survey, according to the following definition. Suppose that  the present and prior surveys are conducted over the categories $C_i, i=1,\ldots,p$ and $C^*_{j},\ j=1,\ldots,I$, respectively. If each $C_j^*$ is a union of some $C_{i}$'s, we say that the prior survey is coarse (compared with the present survey). 

To clarify the relationship of the categories between the two surveys, let us consider the present survey as a two-stage model. We use the two-stage notation $C_{ij},\ i=1,\ldots,I, j=1,\ldots,J_i$ for the categories of the present survey, where for $i=1,\ldots,I$, 
\[
C_{i\cdot} \triangleq \bigcup_{j=1,\ldots,J_i} C_{ij},\qquad  C_{i\cdot}=C^*_i.
\]
Consequently, the first index $i$ of $C_{ij}$ tells to which $C_{i\cdot}=C^*_i$ an observation belongs. The second index $j$ runs over $J_i$ categories that belong to $C_{i\cdot}$. Note that when $J_i=1$, the category $C_{i\cdot}$ consists of a single $C_{i1}$.

The definition of the two-stage multinomial distribution model is described as follows. Suppose that random variable $X$ takes values that belong to one of the categories $C_{ij}$ with the probabilities
\begin{equation}
m_{ij}\triangleq P(X \in C_{ij}),
\end{equation}
for  $i=1,\ldots,I, \ j=1,\ldots, J_i$. 

We define the related probabilities as
\begin{equation}
m_{i\cdot}\triangleq \sum_{j=1}^{J_i} m_{ij}, \qquad p_{ij}\triangleq \frac{m_{ij}}{m_{i\cdot}}
\end{equation}
for  $i=1,\ldots,I, \ j=1,\ldots, J_i$. The first-stage model is given by focusing on which $C_{i\cdot}$ the value of $X$ belongs to. Its parameters (probabilities for each category) are given by 
\[
m_f \triangleq (m_{1\cdot},\ldots,m_{I\cdot}).
\]
In this paper, we only consider the case where the first-stage model is a full model. This illustrates well how the information on the prior survey changes the estimation risk of the present survey.

The $i$-th model ($i=1,\ldots, I$) on the second stage is given by the categories $C_{ij}, j=1,\ldots, J_i$ and the corresponding probabilities 
\[
p_i\triangleq (p_{i1},\ldots,p_{iJ_i})
\]
under the condition $X \in C_{i\cdot}$. We use $s_i$ for the dimension of the $i$-th model, hence if $s_i=J_i-1$, it is a full model. Note that if the $i$-th second-stage model consists of a single category, i.e., $J_i=1$, $p_{i1}=1$, there are no unknown parameters for the model.

Correspondingly to the two-stage model, the two-stage estimator $\widehat{m}_{ij}$ of $m_{ij}$ is given by
\begin{equation}
\label{decomp_mle}
\widehat{m}_{ij}=\widehat{m}_{i\cdot}\widehat{p}_{ij},\qquad i=1,\ldots,I, \ j=1,\ldots, J_i,
\end{equation}
where $\widehat{m}_{i\cdot}$ is the estimator for the first-stage model and $\widehat{p}_{ij}, \  j=1,\ldots, J_i$ is the estimator for the $i$-th ($i=1,\ldots,I$) second-stage model. Note that for the $i$-th second-stage model with $J_i=1$, the estimator $\widehat{p}_{i1}\equiv 1$.

Let  $x_{ij}$ and $ x_{i\cdot}$  $(i=1,\ldots,I, j=1,\ldots,J_i)$ denote the number of observations in the present survey which belong to $C_{ij}$ and $C_{i\cdot}$, respectively. Hence,
\[
x_{i\cdot} = \sum_{j=1}^{J_i}x_{ij},\qquad \sum_{i=1}^{I}x_{i\cdot}=n \text{ (sample size)}.
\]
The prior survey, which is coarse, is only related to the first-stage model. Let $x_i^*$ denote the number of observations in the prior survey that belong to $C_{i\cdot}$ ($i=1,\ldots,I)$. The sample size of the prior survey is denoted by $n^*$. The notations of $X_{ij}, X_{i \cdot}, X^*_i$ are similarly defined as the corresponding random variables. 

With two sets of samples $x_{i\cdot},\ i=1,\ldots,I$ and $x_i^*,\ i=1,\ldots,I$, there are three possible MLEs for the first-stage parameter $m=(m_{1\cdot},\ldots,m_{I\cdot})$. (We use capital letters for the corresponding random variables.)
\begin{enumerate}
\item MLE from the present sample: 
\[
\widehat{m}_f \triangleq (\widehat{m}_{1\cdot},\ldots,\widehat{m}_{I\cdot}),\qquad
\widehat{m}_{i\cdot}=X_{i\cdot}/n,\ i=1,\ldots,I.
\]
\item MLE from the prior sample: 
\[
\widehat{m}_f^* \triangleq (\widehat{m}^*_{1\cdot},\ldots,\widehat{m}^*_{I\cdot}),\qquad
\widehat{m}^*_{i\cdot}=X^*_i/n^*,\ i=1,\ldots,I.
\]
\item MLE from the pooled sample: 
\[
\widehat{m}_f^p \triangleq (\widehat{m}^p_{1\cdot},\ldots,\widehat{m}^p_{I\cdot}), \qquad \widehat{m}^p_{i\cdot} \triangleq \frac{X_{i\cdot}+X_{i}^*}{n+n^*}.
\]
\end{enumerate}
MLE for the $i$-th second-stage model ($i=1,\ldots, I$) based on the present sample is given by
\[
\widehat{p}_{ij} \triangleq X_{ij}/X_{i \cdot},\quad j=1,\ldots, J_i.
\]
If we combine this estimator with each estimator for the first-stage model, 
we have three estimators for $m_{ij} (i=1,\ldots,I, j=1,\ldots J_i)$:
\begin{align}
\label{original_est_first_stage}
\widehat{m}_{ij}&\triangleq \widehat{m}_{i\cdot} \widehat{p}_{ij},\quad\widehat{m}\triangleq (\widehat{m}_{ij})_{1\leq i \leq I, 1\leq j \leq J_i}\\
\label{altern_est_first_stage}
\widehat{m}^*_{ij}&\triangleq \widehat{m}^*_{i\cdot} \widehat{p}_{ij}, \quad\widehat{m}^*\triangleq (\widehat{m}^*_{ij})_{1\leq i \leq I, 1\leq j \leq J_i}\\
\label{pool_est_first_stage}
\widehat{m}^p_{ij}&\triangleq \widehat{m}^p_{i\cdot} \widehat{p}_{ij}, \quad \widehat{m}^p\triangleq (\widehat{m}^p_{ij})_{1\leq i \leq I, 1\leq j \leq J_i}.
\end{align}

This study aims to compare the risks of these three estimators. Asymptotic expansion of the risk for the three estimators is given in Section \ref{subsec: theo_result}. Comparison of the first and second-order terms give us some insight on the asymptotic behavior of the estimators. In Section \ref{subsec:example}, we observe the performance of the estimators through three examples. The findings from our simulation studies supplement the theoretical results.
%
%
%
%
%
%
%
%


%
%
%
%
%
%
%
%
\section{Main results}
\label{sec:main_result}

%
%
%
%
%
\subsection{Theoretical Result}
\label{subsec: theo_result}
In this section, we derive an asymptotic expansion of the risk of three estimators $\widehat{m}$, $\widehat{m}^*$, and $\widehat{m}^p$. We use the following decomposition of Kullback--Leibler divergence $\overset{-1}{D}[\widehat{m}: m]$. Let
\[
m\triangleq (m_{ij})_{1\leq i \leq I, 1\leq j \leq J_i},\qquad \widehat{
m}\triangleq (\widehat{m}_{ij})_{1\leq i \leq I, 1\leq j \leq J_i},
\]
then from \eqref{decomp_mle}, we have
\begin{align}
\overset{-1}{D}[\widehat{m}: m]&=\sum_{i=1}^I \sum_{j=1}^{J_i} \widehat{m}_{ij} \log (\widehat{m}_{ij}/m_{ij})\nonumber\\
&=\sum_{i=1}^I \sum_{j=1}^{J_i} \widehat{m}_{i\cdot}\widehat{p}_{ij} \log \frac{\widehat{m}_{i\cdot}\widehat{p}_{ij}}{m_{i\cdot}p_{ij}}\nonumber\\
&=\sum_{i=1}^I \widehat{m}_{i\cdot} \log (\widehat{m}_{i\cdot}/m_{i\cdot}) +\sum_{i=1}^I \widehat{m}_{i\cdot}\sum_{j=1}^{J_i}\widehat{p}_{ij} \log(\widehat{p}_{ij}/p_{ij})\nonumber\\
&=\overset{-1}{D}[\widehat{m}_f: m_f] + \sum_{i=1}^I \widehat{m}_{i\cdot} \:\overset{-1}{D}[\widehat{p}_i:p_i], \label{chain_rule_-1}
\end{align}
where 
\[
\widehat{m}_f= (\widehat{m}_{1\cdot},\ldots\widehat{m}_{I\cdot}),\qquad \widehat{p}_i\triangleq (\widehat{p}_{i1},\ldots,\widehat{p}_{i J_i}),\quad 1\leq i \leq I.
\]
This decomposition is called the ``chain rule" and is one of the appealing characteristics of the Kullback--Leibler divergence.

As for the estimation of $p_{ij}$, the $i$-th second-stage model parameter, we almost surely encounter the case where $X_{i \cdot}=0$. In this case, MLE (or any estimator) does not work, because there is no sample from the $i$-th model. We make it a rule to discard such samples with no estimation. In the following theorems, we adopt this rule, hence expectations taken are all conditional on the state $X_{i \cdot} \ne 0,\ 1 \leq \forall i \leq I$. However, as Lemma \ref{expectations_difference} in Appendix shows, the conditional expectations of Kullback--Leibler divergence and  $(X_{i\cdot}-n m_{i \cdot})^k (k=1,2,\ldots)$ differ from those unconditional by $o(n^{-s})$ for any $s >0$. Therefore, we use the same notation as that for the unconditional distribution.

The risk for each estimator is given in the following three theorems. We use the notation $A_{(i)}, i=1,\ldots,I$ for the coefficient of $1/(24n^2)$ in the asymptotic expansion of the risk of MLE for the $i$-th second-stage model when the sample size is given as $n$. Because the dimension of the $i$-th model equals $s_i$, the asymptotic expansion of the risk of MLE with respect to $n$ is given by
\begin{equation}
\label{asy_expan_ith_sub}
\overset{-1}{ED}[\widehat{p}_i:p_i]\triangleq E\bigl[\overset{-1}{D}[\widehat{p}_i:p_i]\bigr]=\frac{s_i}{2n}+\frac{1}{24n^2} A_{(i)}+o(n^{-2}).
\end{equation}
(See (25) of \cite{Sheena}.)
%
%
%
%
\begin{theo} If we use the estimator \eqref{original_est_first_stage}, then the risk is given by
\label{ED_-1_2stage}
\begin{equation}
\begin{aligned}
\overset{-1}{ED}[\widehat{m}:m]&= \frac{p'}{2n}+\frac{1}{24n^2}\Bigl\{\sum_{i=1}^I m_{i\cdot}^{-1}\Bigl(\overset{-1}{A}_{(i)}+2\Bigr)-2\Bigr\}+o(n^{-2})\\
&=\frac{1}{2n}\Bigl(I-1+\sum_{i=1}^I s_i\Bigr)+\frac{1}{12n^2}(M_f-1)+\frac{1}{24n^2}\Bigl(\sum_{i=1}^I m_{i\cdot}^{-1}\overset{-1}{A}_{(i)}\Bigr)+o(n^{-2}),
\label{presetn_first_stage}
\end{aligned}
\end{equation}
where 
\[
p'\triangleq I-1+\sum_{i=1}^I s_i, \quad M_f\triangleq \sum_{i=1}^I m_{i\cdot}^{-1}
\]
In the special case, when the second-stage models are all full models, the risk is given by
\begin{equation}
\label{presetn_first_stage_2nd_model_full}
\overset{-1}{ED}[\widehat{m}:m]=\frac{p}{2n}+\frac{1}{12n^2}\Bigl\{\sum_{i=1}^I \sum_{j=1}^{J_i} m_{ij}^{-1}-1\Bigr\}+o(n^{-2}),
\end{equation}
where
\[
p\triangleq \sum_{i=1}^I J_i -1.
\]
\end{theo}
\begin{proof}
From \eqref{chain_rule_-1}, we have
\begin{align}
&\overset{-1}{ED}[\widehat{m}:m]\nonumber\\
&\triangleq E[\overset{-1}{D}[\widehat{m}:m]]\nonumber\\
&=E_{X_{1\cdot},\cdots,X_{I\cdot}} \bigl[E[\overset{-1}{D}[\widehat{m}:m] | X_{1\cdot},\ldots,X_{I\cdot}] \bigr]\nonumber\\
&= E_{X_{1\cdot},\cdots,X_{I\cdot}} \Bigl[E\bigl[\overset{-1}{D}[\widehat{m}_f: m_f] + \sum_{i=1}^I (X_{i\cdot}/n) \:\overset{-1}{D}[\widehat{p}_i:p_i] \big| X_{1\cdot},\ldots,X_{I\cdot}\bigr] \Bigr]\nonumber\\
&=\overset{-1}{ED_f}+\sum_{i=1}^I  E_{X_{1\cdot},\cdots,X_{I\cdot}} [ (X_{i\cdot}/n) \:\overset{-1}{ED}_{(i)}(X_{i\cdot}) ],
\label{decomp_KL_two_stage}
\end{align}
where
\begin{align*}
\overset{-1}{ED_f}&\triangleq E[\overset{-1}{D}[\widehat{m}_f:m_f]],\\
\overset{-1}{ED}_{(i)}(X_{i\cdot})&\triangleq  E[\overset{-1}{D}[\widehat{p}_{i}:p_{i}]]\big|_{n=X_{i\cdot}}.
\end{align*}
Note that in general, $E[\overset{-1}{D}[\widehat{p}_{i}:p_{i}]]$ is the function of the size $n$ of the sample on which $\widehat{p}_{i}$ are based. Here, in the two-stage model, size $n$ equals to $X_{i\cdot}$.

Because
\[
\overset{-1}{ED}_{(i)}(X_{i\cdot})=\frac{s_i}{2X_{i\cdot}}+\frac{1}{24X_{i\cdot}^2}\overset{-1}{A}_{(i)}+o(n^{-2})\big|_{n=X_{i\cdot}}
\]
we have
\begin{align*}
&E_{X_{1\cdot},\cdots,X_{I\cdot}} [ (X_{i\cdot}/n) \:\overset{-1}{ED}_{(i)}(X_{i\cdot}) ]\\
&=\frac{s_i}{2n}+\frac{1}{24n}\overset{-1}{A}_{(i)}E[X_{i\cdot}^{-1}]+E[o_p(n^{-2})].
\end{align*}
For the derivation of the relation
\[
o(n^{-2})\big|_{n=X_{i\cdot}}=o_p(n^{-2}),
\]
we used the Lemma \ref{converge->p_converge}  in Appendix and the fact $X_{i \cdot}/n$ converges to $m_{i\cdot}$ in probability. Taylor expansion of $1/X_{i \cdot}$ around $1/(nm_{i\cdot})$ is given by
\[
\frac{1}{X_{i \cdot}}=\frac{1}{nm_{i\cdot}}-\frac{1}{(nm_{i\cdot})^2}(X_{i \cdot}-nm_{i\cdot})+O_p(n^{-2}).
\]
Because $E[X_{i \cdot}-nm_{i\cdot}]=0$, we have 
\[
E[X_{i \cdot}^{-1}]=\frac{1}{nm_{i\cdot}}+O(n^{-2}).
\]
and 
\begin{equation}
\label{part_result_1_Th_2}
E_{X_{1\cdot},\cdots,X_{I\cdot}} [ (X_{i\cdot}/n) \:\overset{-1}{ED}_{(i)}(X_{i\cdot}) ]=\frac
{s_i}{2n}+\frac{1}{24n^2}m_{i\cdot}^{-1}\overset{-1}{A}_{(i)}+o(n^{-2}).
\end{equation}
(We omit the proof that $E[O_p(n^{-2})]=O(n^{-2})$ and $E[o_p(n^{-2})]=o(n^{-2})$ hold for the present case.) 

Consequently, 
\begin{equation}
\label{last_but_one_Theo1}
\overset{-1}{ED}[\widehat{m}:m]=\overset{-1}{ED_f}+\sum_{i=1}^I \Bigl(\frac
{s_i}{2n}+\frac{1}{24n^2}m_{i\cdot}^{-1}\overset{-1}{A}_{(i)}\Bigr)+o(n^{-2}).
\end{equation}
If we apply \eqref{ED_-1_expan_full} to $\overset{-1}{ED_f}$, we obtain 
\[
\overset{-1}{ED_f}=\frac{I-1}{2n}+\frac{1}{12n^2}(M_f-1)+o(n^{-2}).
\]
If we insert this result into \eqref{last_but_one_Theo1}, we have \eqref{presetn_first_stage}.

When the second-stage models are all full models,  $s_i= J_i-1,\ i=1,\ldots,I$. From \eqref{ED_-1_expan_full}, it turns out that 
\[
\sum_{i=1}^I m^{-1}_{i\cdot}\overset{-1}{A}_{(i)}=\sum_{i=1}^I  m^{-1}_{i\cdot} \Bigl(\sum_{j=1}^{J_i}2p_{ij}^{-1}-2\Bigr)
=2 \sum_{i=1}^I  \sum_{j=1}^{J_i} m_{ij}^{-1}-2M_f.
\]
If we insert these results into \eqref{presetn_first_stage}, we have \eqref{presetn_first_stage_2nd_model_full}. (Because the total model is a full model in this case,
we notice that \eqref{presetn_first_stage_2nd_model_full} could be derived from \eqref{ED_-1_expan_full} straightforwardly.)
\end{proof}
%
%
%
%
\begin{theo}
\label{theorem:ED_-1_eval_prior}
If we use the estimator \eqref{altern_est_first_stage}, then the risk is given by
\begin{equation}
\label{ED_-1_eval_prior}
\begin{aligned}
\overset{-1}{ED}[\widehat{m}^*:m]&= \frac{1}{2n^*}(I-1)+\frac{1}{2n}\sum_{i=1}^I s_i+
\frac{1}{12{n^*}^2}(M_f-1)\\
&\qquad+\frac{1}{24n^2}\sum_{i=1}^I m_{i\cdot}^{-1}\Bigl(\overset{-1}{A}_{(i)}+12(1-m_{i\cdot})s_i\Bigr)+o(n^{-2})+o({n^*}^{-2}). 
\end{aligned}
\end{equation}
In the special case, when the second-stage models are all full models, the risk is given by
\begin{equation}
\label{ED_-1_eval_prior_2nd_model_full}
\begin{aligned}
\overset{-1}{ED}[\widehat{m}^*:m]&= \frac{1}{2n^*}(I-1)+\frac{1}{2n}(p+1-I)\\
&\qquad+
\frac{1}{12{n^*}^2}(M_f-1)+
\frac{1}{12n^2}\Bigl(M+\sum_{i=1}^I m_{i\cdot}^{-1}(6J_i-7)-6(p+1-I)\Bigr)\\&
\qquad+o(n^{-2})+o({n^*}^{-2}). 
\end{aligned}
\end{equation}
\end{theo}
\begin{proof}
From \eqref{chain_rule_-1}, we have
\begin{align*}
&\overset{-1}{ED}[\widehat{m}^*:m]\\
&\triangleq E[\overset{-1}{D}[\widehat{m}^*:m]]\\
&=E_{X^*_{1},\cdots,X^*_{I}}\Bigl[E_{X_{1\cdot},\cdots,X_{I\cdot}} \bigl[E[\overset{-1}{D}[\widehat{m}^*:m] | X_{1\cdot},\ldots,X_{I\cdot}] \big| X^*_{1},\cdots,X^*_{I}\bigr]\Bigr]\\
&= E_{X^*_{1},\cdots,X^*_{I}}\Bigl[E_{X_{1\cdot},\cdots,X_{I\cdot}} \Bigl[E\bigl[\overset{-1}{D}[\widehat{m}^*_f: m_f] + \sum_{i=1}^I (X^*_{i}/n^*) \:\overset{-1}{D}[\widehat{p}_i:p_i] \big| X_{1\cdot},\ldots,X_{I\cdot}\bigr] \Big| X^*_{1},\cdots,X^*_{I}\Bigr]\Bigr]\\
&=\overset{-1}{ED_f^*}+\sum_{i=1}^I  E_{X^*_{1},\cdots,X^*_{I}}\bigl[E_{X_{1\cdot},\cdots,X_{I\cdot}} [ (X^*_{i}/n^*) \:\overset{-1}{ED}_{(i)}(X_{i\cdot}) | X^*_{1},\cdots,X^*_{I} ]\bigr],
\end{align*}
where
\begin{align*}
\overset{-1}{ED_f^*}&\triangleq E[\overset{-1}{D}[\widehat{m}^*_f:m_f]],\\
\overset{-1}{ED}_{(i)}(X_{i\cdot})&\triangleq  E[\overset{-1}{D}[\widehat{p}_{i}:p_{i}]]\big|_{n=X_{i\cdot}}.
\end{align*}

Because
\[
\overset{-1}{ED}_{(i)}(X_{i\cdot})=\frac{s_i}{2X_{i\cdot}}+\frac{1}{24X_{i\cdot}^2}\overset{-1}{A}_{(i)}+o(n^{-2})\big|_{n=X_{i\cdot}}
\]
we have
\begin{align*}
&E_{X^*_{1},\cdots,X^*_{I}}\bigl[E_{X_{1\cdot},\cdots,X_{I\cdot}} [ (X^*_{i}/n^*) \:\overset{-1}{ED}_{(i)}(X_{i\cdot}) | X^*_{1},\cdots,X^*_{I} ]\bigr]\\
&=\frac{s_i}{2n^*}E[X^*_{i}]E[X_{i \cdot}^{-1}]+\frac{1}{24n^*}E[X^*_{i}]\overset{-1}{A}_{(i)}E[X_{i\cdot}^{-2}]+E[o_p(n^{-2})]\\
&=\frac{s_i m_{i\cdot}}{2}E[X_{i \cdot}^{-1}]+\frac{m_{i\cdot}}{24}\overset{-1}{A}_{(i)}E[X_{i\cdot}^{-2}]+E[o_p(n^{-2})]
\end{align*}
Taylor expansions of $1/X_{i\cdot}$ and $1/X^2_{i\cdot}$ are given by
\begin{align*}
\frac{1}{X_{i \cdot}}&=\frac{1}{nm_{i\cdot}}-\frac{1}{(nm_{i\cdot})^2}(X_{i \cdot}-nm_{i\cdot})+\frac{1}{(nm_{i\cdot})^3}(X_{i \cdot}-nm_{i\cdot}) ^2+ O_p(n^{-5/2})\\
\frac{1}{X^2_{i \cdot}}&=\frac{1}{(nm_{i\cdot})^2}+ O_p(n^{-5/2}).
\end{align*}
Using $E[X_{i \cdot}-nm_{i\cdot}]=0$ and $E[(X_{i \cdot}-nm_{i\cdot})^2]=nm_{i\cdot}(1-m_{i\cdot})$, we have
\begin{align}
\label{expec_X^-1}
E[X_{i \cdot}^{-1}]&=\frac{1}{nm_{i\cdot}}+\frac{1-m_{i\cdot}}{(nm_{i\cdot})^2}+O(n^{-5/2})\\
\label{expec_X^-2}
E[X_{i \cdot}^{-2}]&=\frac{1}{(nm_{i\cdot})^2}+O(n^{-5/2})
\end{align}
(We omit the proof that $E[O_p(n^{-5/2})]=O(n^{-5/2})$ holds for the present case.) 
From \eqref{expec_X^-1} and \eqref{expec_X^-2}, we have 
\begin{align*}
&E_{X^*_{1},\cdots,X^*_{I}}\bigl[E_{X_{1\cdot},\cdots,X_{I\cdot}} [ (X^*_{i}/n^*) \:\overset{-1}{ED}_{(i)}(X_{i\cdot}) | X^*_{1},\cdots,X^*_{I} ]\bigr]\\
&=\frac{s_i}{2n}+\frac{1}{24n^2} m_{i\cdot}^{-1}\bigl(12 s_i (1-m_{i\cdot})+\overset{-1}{A}_{(i)}\bigr)+o(n^{-2})
\end{align*}
(We omit the proof that $E[o_p(n^{-2})]=o(n^{-2})$ holds for the present case.) 

Consequently, 
\begin{equation}
\label{first_stage_alter_mid}
\overset{-1}{ED}[\widehat{m}^*:m]=\overset{-1}{ED_f^*}+\sum_{i=1}^I \Bigl(\frac
{s_i}{2n}+\frac{1}{24n^2}m_{i\cdot}^{-1}\bigl(12 s_i(1-m_{i\cdot}) +\overset{-1}{A}_{(i)} \bigr)\Bigr)+o(n^{-2}).
\end{equation}
If we apply \eqref{ED_-1_expan_full} to $\overset{-1}{ED_f^*}$, we obtain 
\[
\overset{-1}{ED_f^*}=\frac{I-1}{2n^*}+\frac{1}{12(n^*)^2}(M_f-1)+o((n^*)^{-2}).
\]
If we insert this result in the right-hand side of \eqref{first_stage_alter_mid}, we obtain \eqref{ED_-1_eval_prior}.

By inserting $s_i=J_i-1$ and
\[
\overset{-1}{A}_{(i)}= 2 m_{i\cdot} \sum_{j=1}^{J_i} m_{ij}^{-1}-2
\]
into \eqref{ED_-1_eval_prior}, we obtain \eqref{ED_-1_eval_prior_2nd_model_full}
\end{proof}

We compare the risks of the two estimators $\widehat{m}$ and $\widehat{m}^*$.
From \eqref{presetn_first_stage} and \eqref{ED_-1_eval_prior}, we notice that
\begin{equation}
\label{difference_pre_present_survey}
\begin{aligned}
&\overset{-1}{ED}[\widehat{m}:m]-\overset{-1}{ED}[\widehat{m}^*:m]\\
&=\frac{I-1}{2}\Bigl(\frac{1}{n}-\frac{1}{n^*}\Bigr)+\frac{M_f-1}{12}\Bigl(\frac{1}{n^2}-\frac{1}{(n^*)^2}\Bigr)-\frac{1}{2n^2}\Bigl(\sum_{i=1}^I s_i m_{i\cdot}^{-1}(1-m_{i\cdot})\Bigr)
\\
&\qquad+o(n^{-2})+o({n^*}^{-2})\\
&=\frac{I-1}{2n}\frac{\beta-1}{\beta}+\frac{1}{12n^2}\Bigl(\frac{\beta^2-1}{\beta^2}(M_f-1)-6\sum_{i=1}^I s_im_{i\cdot}^{-1}(1-m_{i\cdot})\Bigr)\\
&\qquad+o(n^{-2}),\\
\end{aligned}
\end{equation}
where $\beta \triangleq n^*/n$. 

This equation gives us the following suggestions on whether we should use the prior information or not:
\begin{itemize}
\item The difference in the first-order term, $(I-1)(n^{-1}-{n^*}^{-1})/2$,  leads us to the simple conclusion that if $n^* > n$ and both are large enough, it is better to use $\widehat{m}^*$.  
\item To see the effect of the difference in the second-order term, we let  $n^*=n$, then the difference equals
\begin{equation}
\label{difference_pre_present_survey_sec_term}
-\frac{1}{2n^2}\Bigl(\sum_{i=1}^I s_i(m_{i\cdot}^{-1}-1)\Bigr) <0.
\end{equation}
This suggests that if the prior survey is done in a similar scale ($n^* \doteqdot n$), then it is better to use $\widehat{m}$. The ``unified'' estimation using the same sample for both stages has an advantage over the ``non-unified'' estimation that uses the different sample sets for the two stages. (We confirm this observation in the next subsection by simulation.) Notice that in \eqref{chain_rule_-1}, each $\overset{-1}{D}[\widehat{p}_i:p_i]$ is multiplied by $\widehat{m}_{i\cdot}$. If $x_{i\cdot}$, the sample size used for the $i$-th second-stage model estimation, is small, then $\widehat{m}_{i\cdot}=x_{i\cdot}/n$ also gets small. Hence, the loss $\overset{-1}{D}[\widehat{p}_i:p_i]$ is devalued in the total loss evaluation $\overset{-1}{D}[\widehat{m}: m]$ in the unified approach \eqref{original_est_first_stage}. This is natural because when the sample size is small, we should not rely much on the result. This mechanism does not work if we use independent sets of samples for the first and second stages as in the non-unified approach. From \eqref{difference_pre_present_survey_sec_term}, we notice that this negative effect of non-unified approach is strong when the first-stage model contains a category of small probability (i.e., $m_{i\cdot}^{-1}$ is large) and especially if the corresponding second-stage model has many parameters to be estimated (i.e., $s_i$ is large). 
\end{itemize}

The risk for the estimation using $\widehat{m}^p$ is given by the following theorem.
%
%
%
%
\begin{theo}
\label{theory:ED_-1_eval_prior_b}
If we use the estimator \eqref{pool_est_first_stage}, then the risk is given by
\begin{equation}
\label{ED_-1_eval_prior_b}
\begin{aligned}
\overset{-1}{ED}[\widehat{m}^p:m]&= \frac{1}{2(n+n^*)}(I-1)+\frac{1}{2n}\sum_{i=1}^I s_i+
\frac{1}{12{(n+n^*)}^2}(M_f-1)\\
&\quad+\frac{1}{24n^2}\sum_{i=1}^I m_{i\cdot}^{-1}\Bigl(\overset{-1}{A}_{(i)}+12(1-m_{i\cdot})s_i n^*/(n+n^*)\Bigr)\\
&\quad+o(n^{-2})+o((n+n^*)^{-2}). 
\end{aligned}
\end{equation}
In the special case, when the second-stage models are all full models, the total risk is given by
\begin{equation}
\label{ED_-1_eval_prior_b_2nd_full}
\begin{aligned}
\overset{-1}{ED}[\widehat{m}^p:m]&= \frac{1}{2(n+n^*)}(I-1)+\frac{1}{2n} (p-I+1)+
\frac{1}{12{(n+n^*)}^2}(M_f-1)\\
&\quad+\frac{1}{12n^2}\Bigl(M-M_f+ \frac{6n^*}{n+n^*}\Bigl(\sum_{i=1}^I m_{i\cdot}^{-1}J_i-M_f-(p-I+1)\Bigr) \Bigr)\\
&\quad+o(n^{-2})+o(((n+n^*)^{-2}). 
\end{aligned}
\end{equation}
\end{theo}
\begin{proof}
As in \eqref{decomp_KL_two_stage}, we have
\begin{align*}
&\overset{-1}{ED}[\widehat{m}^p:m]\\
&\triangleq E[\overset{-1}{D}[\widehat{m}^p:m]]\\
&=E_{X^*_{1},\cdots,X^*_{I}}\Bigl[E_{X_{1\cdot},\cdots,X_{I\cdot}} \bigl[E[\overset{-1}{D}[\widehat{m}^p:m] | X_{1\cdot},\ldots,X_{I\cdot}] \big| X^*_{1},\cdots,X^*_{I}\bigr]\Bigr]\\
&= E_{X^*_{1},\cdots,X^*_{I}}\Bigl[E_{X_{1\cdot},\cdots,X_{I\cdot}} \Bigl[E\bigl[\overset{-1}{D}[\widehat{m}^p_f: m_f] \\
&\hspace{45mm}+ \sum_{i=1}^I(X_{i\cdot}+X^*_{i})/(n+n^*) \:\overset{-1}{D}[\widehat{p}_i:p_i] \big| X_{1\cdot},\ldots,X_{I\cdot}\bigr] \Big| X^*_{1},\cdots,X^*_{I}\Bigr]\Bigr]\\
&=\overset{-1}{ED_f^p}+\sum_{i=1}^I  E_{X^*_{1},\cdots,X^*_{I}}\bigl[E_{X_{1\cdot},\cdots,X_{I\cdot}} [ (X_{i\cdot}+X^*_{i})/(n+n^*) \:\overset{-1}{ED}_{(i)}(X_{i\cdot}) | X^*_{1},\cdots,X^*_{I} ]\bigr],
\end{align*}
where
\begin{align*}
\overset{-1}{ED_f^p}&\triangleq E[\overset{-1}{D}[\widehat{m}^p_f:m_f]],\\
\overset{-1}{ED}_{(i)}(X_{i\cdot})&\triangleq  E[\overset{-1}{D}[\widehat{p}_{i}:p_{i}]]\big|_{n=X_{i\cdot}}.
\end{align*}

Because
\[
\overset{-1}{ED}_{(i)}(X_{i\cdot})=\frac{s_i}{2X_{i\cdot}}+\frac{1}{24X_{i\cdot}^2}\overset{-1}{A}_{(i)}+o(n^{-2})\big|_{n=X_{i\cdot}}
\]
we have
\begin{align*}
&E_{X^*_{1},\cdots,X^*_{I}}\bigl[E_{X_{1\cdot},\cdots,X_{I\cdot}} [ (X_{i\cdot}+X^*_{i})/(n+n^*) \:\overset{-1}{ED}_{(i)}(X_{i\cdot}) | X^*_{1},\cdots,X^*_{I} ]\bigr]\\
&=\frac{s_i}{2(n+n^*)}E[X^*_{i}]E[X_{i \cdot}^{-1}]+\frac{s_i}{2(n+n^*)}+\frac{1}{24(n+n^*)}\overset{-1}{A}_{(i)}E[X^*_{i}]E[X_{i\cdot}^{-2}]\\
&\qquad+\frac{1}{24(n+n^*)}\overset{-1}{A}_{(i)}E[X_{i\cdot}^{-1}]+E[o_p(n^{-2})]\\
&=\frac{s_i m_{i\cdot}}{2}\frac{n^*}{n+n^*}E[X_{i \cdot}^{-1}]+\frac{s_i}{2(n+n^*)}+\frac{m_{i\cdot}}{24}\frac{n^*}{n+n^*}\overset{-1}{A}_{(i)}E[X_{i\cdot}^{-2}]\\
&\qquad+\frac{1}{24(n+n^*)}\overset{-1}{A}_{(i)}E[X_{i\cdot}^{-1}]+E[o_p(n^{-2})]
\end{align*}
Because of \eqref{expec_X^-1} and \eqref{expec_X^-2}, we have 
\begin{align*}
&E_{X^*_{1},\cdots,X^*_{I}}\bigl[E_{X_{1\cdot},\cdots,X_{I\cdot}} [ (X_{i\cdot}+X^*_{i})/(n+n^*) \:\overset{-1}{ED}_{(i)}(X_{i\cdot}) | X^*_{1},\cdots,X^*_{I} ]\bigr]\\
&=\frac{s_i}{2n}+\frac{1}{24n^2} m_{i\cdot}^{-1}\Bigl(12 s_i (1-m_{i\cdot})\frac{n^*}{n+n^*}+\overset{-1}{A}_{(i)}\Bigr)+o(n^{-2})
\end{align*}
(We omit the proof that $E[o_p(n^{-2})]=o(n^{-2})$ holds for the present case.) 

Consequently,
\begin{equation}
\label{first_stage_pool_mid}
\begin{aligned}
\overset{-1}{ED}[\widehat{m}^p:m]&=\overset{-1}{ED_f^p}+\sum_{i=1}^I \Bigl(\frac
{s_i}{2n}+\frac{1}{24n^2}m_{i\cdot}^{-1}\bigl(12 s_i(1-m_{i\cdot})n^*/(n+n^*) +\overset{-1}{A}_{(i)} \bigr)\Bigr)\\
&\quad+o(n^{-2}).
\end{aligned}
\end{equation}
If we apply \eqref{ED_-1_expan_full} to $\overset{-1}{ED_f^p}$, we obtain 
\[
\overset{-1}{ED_f^p}=\frac{I-1}{2(n+n^*)}+\frac{1}{12(n+n^*)^2}(M_f-1)+o((n+n^*)^{-2}).
\]
If we insert this result into \eqref{first_stage_pool_mid}, we obtain \eqref{ED_-1_eval_prior_b}.

By inserting $s_i=J_i-1$ and
\[
\overset{-1}{A}_{(i)}= 2 m_{i\cdot} \sum_{j=1}^{J_i} m_{ij}^{-1}-2
\]
into \eqref{ED_-1_eval_prior_b}, we obtain \eqref{ED_-1_eval_prior_b_2nd_full}.
\end{proof}
Let us compare the risk of $\widehat{m}^p$ with that of $\widehat{m}$.  From \eqref{presetn_first_stage} and \eqref{ED_-1_eval_prior_b}, we notice that
\begin{equation}
\label{difference_present_pool}
\begin{aligned}
&\overset{-1}{ED}[\widehat{m}:m]-\overset{-1}{ED}[\widehat{m}^p:m]\\
&=\frac{I-1}{2}(n^{-1}-(n+n^*)^{-1})+\frac{M_f-1}{12}(n^{-2}-(n+n^*)^{-2})-\frac
{n^*}{2n^2(n+n^*)}\sum_{i=1}^I m_{i\cdot}^{-1}(1-m_{i\cdot})s_i\\
&\quad+o(n^{-2})+o((n+n^*)^{-2})\\
&=\frac{I-1}{2n}\frac{\beta}{1+\beta}+\frac{1}{n^2}\frac{\beta^2+2\beta}{(1+\beta)^2}\frac{M_f-1}{12}-\frac{1}{2n^2}\frac{\beta}{1+\beta}\sum_{i=1}^I m_{i\cdot}^{-1}(1-m_{i\cdot})s_i\\
&\quad+o(n^{-2})+o((n+n^*)^{-2}).
\end{aligned}
\end{equation}
We observe the following points:
\begin{itemize}
\item From the first-order term of \eqref{difference_present_pool}, it is easily confirmed that $\widehat{m}^p$ is better than $\widehat{m}$ with large enough size of $n$ and $n^*$. 
\item Taking the second-order terms into account, \eqref{difference_present_pool} could be negative. This could happen when large $s_i$ exists, that is, some of the second-stage models contain a large number of parameters. For example, let $m_{1\cdot}=\cdots =m_{I\cdot}=1/I$ and suppose that $n^*$ is relatively so large to $n$ that $\beta/(1+\beta)\doteqdot 1$ and $(\beta^2+\beta)/(1+\beta)^2 \doteqdot 1$, then \eqref{difference_present_pool} approximately equals 
\[
(I-1)\Bigl(\frac{1}{2n}+\frac{I+1}{12n^2}-\frac{1}{2n^2}\sum_{i=1}^I s_i\Bigr)
\]
up to the second-order terms. This is negative when $n < \sum_{i=1}^I s_i -(I+1)/6$. Because we are neglecting the higher-order terms, we are unsure if the situation $\overset{-1}{ED}[\widehat{m}:m]<\overset{-1}{ED}[\widehat{m}^p:m]$ could happen with some $n$. (We confirm this later by simulation in Section \ref{subsec:example}.) Anyway, we notice that superiority of the pooled estimator fades when the second-stage models are highly complicated.
\item As in the estimator $\widehat{m}^*$, the use of another set of sample for the first-stage model estimation has an adverse effect on the risk. To see this phenomenon, consider the case that the present sample size equals to $n+n^*(\triangleq \tilde{n})$. In this case, \eqref{presetn_first_stage} with $n$ substituted with $\tilde{n}$ turns out 
\begin{equation}
\label{presetn_first_stage_ntilde} 
\begin{aligned}
&\overset{-1}{ED}[\widehat{m}:m](\tilde{n})\\
&=\frac{1}{2\tilde{n}}\Bigl(I-1+\sum_{i=1}^I s_i\Bigr)+\frac{1}{12\tilde{n}^2}(M_f-1)+\frac{1}{24\tilde{n}^2}\Bigl(\sum_{i=1}^I m_{i\cdot}^{-1}\overset{-1}{A}_{(i)}\Bigr)+o(\tilde{n}^{-2}), 
\end{aligned}
\end{equation}
and if $\overset{-1}{A}_{(i)} >0,\ i=1,\ldots,I$, then neglecting the higher-order terms, we have
\begin{equation}
\begin{aligned}
&\overset{-1}{ED}[\widehat{m}:m](\tilde{n})-\overset{-1}{ED}[\widehat{m}^p:m](n,n^*)\\
&=-\frac{1}{2n}\frac{\beta}{1+\beta}\sum_{i=1}^I s_i-\frac{1}{24n^2} \frac{\beta^2+2\beta}{(1+\beta)^2}\sum_{i=1}^I m_{i\cdot}^{-1}\overset{-1}{A}_{(i)}\\
&\qquad -\frac{1}{2n^2}\frac{\beta}{1+\beta}\sum_{i=1}^I m_{i\cdot}^{-1}(1-m_{i\cdot})s_i <0.
\end{aligned}
\end{equation}
Under the condition of the same total sample size $\tilde{n}$, unified approach $\widehat{m}$ is superior to non-unified approach $\widehat{m}^p$.
\end{itemize}

As for the comparison between $\widehat{m}^*$ and $\widehat{m}^p$, it is obvious from
\eqref{ED_-1_eval_prior_b} and \eqref{ED_-1_eval_prior} that the latter is always asymptotically (with respect to both of the first and the second-order terms) better than the former. 
%
%
%
%
%
%
%
\subsection{Examples}
\label{subsec:example}

We consider three examples to check the theoretical results in \ref{subsec: theo_result} for the comparison of three estimators $\widehat{m}$, $\widehat{m}^*$, and $\widehat{m}^p$. Each of three examples is a multinomial distribution over a contingency table, where there exists a prior survey on the marginal distribution over the column categories.

First, we use an artificial example to illustrate some points mentioned in the previous subsection. We focus on the situation when the sample size $n$ is relatively small. It is rather impractical setting, but wherein the second-order term works in a non-negligible way and interesting (rather pathological) behavior of the three estimators can be watched. 

The next two examples are the distributions derived from real datasets. We will observe the behavior of the three estimators under practical conditions, where $n$ and $n^*$ are of moderate size. We can confirm the relative performance of the three estimators that is expected from the theoretical result on their asymptotic risks.

We introduce the common notations among the examples. The abbreviations ``pre.$*$'', ``pri.$*$,``pool.$*$'' correspond to $\widehat{m}$, $\widehat{m}^*$, $\widehat{m}^p$, respectively. The notation``$*$.sim''  shows the risk (or the ``required sample size'' defined later) as obtained by simulation. We generated samples for the present and prior surveys with the sizes of given $n$ and $n^*$, respectively, and calculated the loss of the estimators. Repeating this $10^4 (or 10^6)$ times, we obtained the simulated risk by averaging the losses. We also calculated the approximated risk from \eqref{presetn_first_stage}, \eqref{ED_-1_eval_prior}, and  \eqref{ED_-1_eval_prior_b} by neglecting the terms $o(n^{-2})$ and $o((n^*)^{-2})$, which are denoted by ``$*$.app''. 

The ``required sample size (r.s.s.)'' is defined as follows, according to estimators to be compared. Let the risks of the three estimators as the functions of  $n$ and $n^*$ be denoted by $\overset{-1}{ED}[\widehat{m}:m](n)$, $\overset{-1}{ED}[\widehat{m}^*:m](n,n^*)$, $\overset{-1}{ED}[\widehat{m}^p:m](n,n^*)$.\\
1. Comparison between $\widehat{m}$ and $\widehat{m}^*$ through r.s.s. The ``r.s.s for $\widehat{m}^*$ to $\widehat{m}$ under the condition $n=n_0$'' is the solution of $n^*$ for the equation
\begin{equation}
\label{def_rss}
\overset{-1}{ED}[\widehat{m}^*:m](n_0,n^*)=\overset{-1}{ED}[\widehat{m}:m](n_0).
\end{equation}
As we mentioned in the previous subsection, if $n^*=n_0$, then (at least asymptotically) the left side is larger than the right side in \eqref{def_rss}. Therefore, $n^*$ must be larger than $n_0$ so that the equation holds. The r.s.s. shows the degree of inefficiency of $\widehat{m}^*$ (non-unified approach) to $ \widehat{m}$ (unified approach).\\
2. Comparison between $\widehat{m}$ and $\widehat{m}^p$ through r.s.s.  The ``r.s.s for $\widehat{m}$ to $\widehat{m}^p$ under the condition $(n,n^*)=(n_0,n_0^*)$'' is the solution of $n$ for the equation
\begin{equation}
\label{def_rss_2}
\overset{-1}{ED}[\widehat{m}:m](n)=\overset{-1}{ED}[\widehat{m}^p:m](n_0,n_0^*).
\end{equation}
In the previous subsection, we noticed that (at least asymptotically) if $n=n_0+n_0^*$, then the left-hand side is smaller than the right-hand side. Therefore, the solution $n$ for the equation is smaller than $n_0+n_0^*$. The r.s.s. indicates the degree of  efficiency of $\widehat{m}$ to  $\widehat{m}^p$ in view of the total sample size. 

Note that the risks that appear in \eqref{def_rss} and \eqref{def_rss_2} are calculated from either simulation or approximation.
%
%
%
%
\\
\\
\textit{Example 1: 100 by 2 Contingency Table}\\
 Consider a 100 by 2 contingency table and suppose that the distribution over the table  is uniform, that is, each category (cell) has the probability $1/200$. Suppose that we have a prior survey on the two column categories, hence $I=2$ and $J_1=J_2=100$. The second- stage model is a full model over the row categories with a column fixed.

Table \ref{table:ex1_risk} shows the risks of the three estimators calculated from both simulation and approximation, where $*$.sim is the average of $10^6$ simulated losses (its standard deviation is at most $6.13E-6$).

 The r.s.s. of $\widehat{m}^*$ to $\widehat{m}$ under the condition $n=n_0$ is given in Table \ref{table:ex1_rss_pri} with several values of $n_0$. Table \ref{table:ex1_rss_pool} shows the r.s.s. of $\widehat{m}$ to $\widehat{m}^p$ under the condition $(n,n^*)=(n_0,n_0^*)$ with several values $(n_0, n_0^*)$ satisfying $n_0=n_0^*$. 

We can observe the following points from this result. 
\begin{itemize}
\item We confirm that $\widehat{m}$ outperforms $\widehat{m}^*$ under same sample size $n=n^*$. (See the rows for $(n,n^*)=(200,200),\ldots,(1000,1000)$ of Table \ref{table:ex1_risk}.)  When $n$ is small as $n=100,150,200$, even $n^*$ as large as $10^5$ is not enough to compensate the inferiority of non-unified approach. (See the first three rows of Table \ref{table:ex1_risk}.) The r.s.s. of Table \ref{table:ex1_rss_pri} shows more clearly the inferiority of $\widehat{m}^*$ to $\widehat{m}$. Here, $\widehat{m}^*$ needs approximately $110\%-250\%$ larger sample size compared with $\widehat{m}$ to attain the same risk.
\item It is confirmed that $\widehat{m}^p$ is outperformed by $\widehat{m}$ under some  situations like $(n, n^*)=(100,10^5), (150,10^5), (200,10^5), (200, 200)$ in Table \ref{table:ex1_risk}. The relation $\overset{-1}{ED}[\widehat{m}:m](n+n^*)<\overset{-1}{ED}[\widehat{m}^p:m](n,n^*)$ is also confirmed from the fact $\overset{-1}{ED}[\widehat{m}:m](400)=0.283618 < 0.571507=\overset{-1}{ED}[\widehat{m}^p:m](200,200)$ or $\overset{-1}{ED}[\widehat{m}:m](800)=0.132419 < 0.283377=\overset{-1}{ED}[\widehat{m}^p:m](400,400)$. The r.s.s. in Table 
\ref{table:ex1_rss_pool} shows that $\widehat{m}$ needs only 1 to 5 more observations than $n_0^*$ to be par with $\widehat{m}^p$ with the sample size of $(n_0^*, n_0^*)$. 
\item The last 10 rows of Table \ref{table:ex1_risk} shows an interesting aspect of the risks of $\widehat{m}^*$ and $\widehat{m}^p$. As $n$ is fixed to be 90 and $n^*$ increases from $100$ to $1000$, the risk of $\widehat{m}^*$ decreases, while the risk of $\widehat{m}^p$ increases. Increasing the sample size $n^*$ naturally reduces the risk for $\widehat{m}^*$. On the contrary, in the case of $\widehat{m}^p$, the adverse effect of blending two sets of samples outweigh the positive effect of increasing the sample size.
\item The superiority of $\widehat{m}^p$ to $\widehat{m}^*$ is confirmed from the simulated risk.
\item The discrepancy between the simulated risk ($*$.sim) and the approximated risk ($*$.app) could be large when $n$ is small. However, as for the order of the three estimators' risks,  the orders in the simulation and approximation coincide in most cases in Table \ref{table:ex1_risk}.  
\end{itemize}
\begin{table}
\caption{Example 1. Risk of Estimators}
\label{table:ex1_risk}
\footnotesize
\centering
\begin{tabular}[t]{c|c|c|c|c|c|c|c}
$n$ &  $n^*$ & pre.sim  & pri.sim & pool.sim  & pre.app & pri.app & pool.app \\ \hline
100  & 100000  & 1.004442  & 1.006667  & 1.006660  & 1.328325  & 1.333205  & 1.333195  \\ 
\hline
150  & 100000  & 0.734960  & 0.735729  & 0.735722  & 0.811478  & 0.812538  & 0.812532  \\
\hline
200  & 100000  & 0.571438  & 0.571580  & 0.571574  & 0.580831  & 0.580805  & 0.580800  \\
\hline
250  & 100000  & 0.462151  & 0.461980  & 0.461976  & 0.451332  & 0.450917  & 0.450913  \\
\hline
300  & 100000  & 0.384632  & 0.384295  & 0.384291  & 0.368703  & 0.368138  & 0.368135  \\
\hline
200  & 200  & 0.571440  & 0.574083  & 0.571507  & 0.580831  & 0.583306  & 0.580814  \\
\hline
400  & 400  & 0.283618  & 0.284389  & 0.283377  & 0.269583  & 0.270202  & 0.269266  \\
\hline
600  & 600  & 0.181578  & 0.181912  & 0.181328  & 0.175092  & 0.175367  & 0.174813  \\
\hline
800  & 800  & 0.132419  & 0.132599  & 0.132196  & 0.129583  & 0.129738  & 0.129348  \\
\hline
1000  & 1000  & 0.104159  & 0.104270  & 0.103964  & 0.102833  & 0.102932  & 0.102633  \\
\hline
90 & 100 & 1.081035  & 1.088799  & 1.082469  & 1.517068  & 1.528729  & 1.520553  \\ \hline
90 & 200 & 1.081041  & 1.086287  & 1.082924  & 1.517068  & 1.526210  & 1.521638  \\ \hline
90 & 300 & 1.081037  & 1.085444  & 1.083137  & 1.517068  & 1.525373  & 1.522167  \\ \hline
90 & 400 & 1.081040  & 1.085031  & 1.083272  & 1.517068  & 1.524955  & 1.522480  \\ \hline
90 & 500 & 1.081017  & 1.084757  & 1.083334  & 1.517068  & 1.524705  & 1.522687  \\ \hline
90 & 600 & 1.081050  & 1.084623  & 1.083428  & 1.517068  & 1.524538  & 1.522835  \\ \hline
90 & 700 & 1.081047  & 1.084500  & 1.083470  & 1.517068  & 1.524418  & 1.522945  \\ \hline
90 & 800 & 1.081041  & 1.084406  & 1.083501  & 1.517068  & 1.524329  & 1.523030  \\ \hline
90 & 900 & 1.081037  & 1.084332  & 1.083525  & 1.517068  & 1.524260  & 1.523098  \\ \hline
90 & 1000 & 1.081046  & 1.084285  & 1.083557  & 1.517068  & 1.524204  & 1.523153  \\ \hline
\end{tabular}
\caption{Example 1. R.s.s. of $\widehat{m}^*$ to $\widehat{m}$}
\label{table:ex1_rss_pri}
\centering
\begin{tabular}[t]{c|c|c|c|c|c|c|c|c|c}
$n_0$ & 400 & 600 & 800 & 1000 & 1200 & 1400 & 1600 & 1800 & 2000\\
\hline
r.s.s.sim &  995 &  1014 &  1142 & 1296 & 1467 & 1658 & 1854 &2034 & 2229\\
\hline 
r.s.s.app & 791 &  895  & 1063  & 1247 & 1437 & 1631 &1826 &2023 &2220 \\
\hline
\end{tabular}
\caption{Example 1. R.s.s. of $\widehat{m}$ to $\widehat{m^p}$}
\label{table:ex1_rss_pool}
\centering
\begin{tabular}[t]{c|c|c|c|c|c|c|c|c|c}
$n_0=n^*_0$ & 400 & 600 & 800 & 1000 & 1200 & 1400 & 1600 & 1800 & 2000 \\
\hline
r.s.s.sim           & 401  & 601  & 801  & 1002  & 1203  & 1403  & 1604  & 1804  & 2005  \\
\hline
r.s.s.app           & 401  & 601  & 801  & 1002  & 1202  & 1403  & 1603  & 1804  & 2004  \\
\hline
\end{tabular}
\end{table}
%
%
%
%
\bigskip
\textit{Example 2: Breast cancer classification}\\
 The second example is the the breast cancer data from "UCI machine learning repository" (https://archive.ics.uci.edu/ml/datasets/Breast+Cancer). We made a cross-tabulation w.r.t. the variables ``age group'' (5 groups: 30--39, ... ,70--79) and `` the degree of malignancy'' (3 levels: 1, 2, 3) by excluding the only person that is in his/her twenties from the original dataset. By dividing each cell by 285 (the total number of observations), we obtained the relative frequency (see Table \ref{table:breast_cancer_prob}). We suppose that they are the true probability $m_{ij}$ for the categories $C_{ij}$ for $ i=1,\ldots,5, j=1,\ldots,3$ (the first and second index represent the age group and the malignancy, respectively).  

We consider a two-stage model, where the first-stage model is the full model over the age groups, and each of the second-stage models is a full model over malignancy levels within an age group. Hence, $I=5$ and $J_i=3,\ i=1,\ldots 5$. The marginal probability (the column sum in the table) of each age group corresponds to $m_{i\cdot},\ i=1,\ldots,5$, the first-stage parameter. Suppose that we have a prior survey on the ratio of age groups of the sample size $n^*$. 

For several values of $(n,n^*)$, we calculated the approximated risks and the simulated risks. The simulated risk is the average of $10^4$ losses calculated from each generated sample. The result is given in Table \ref{Table:three_est_risk_cancer}, where the simulated risk is in the parenthesis. We calculated two kinds of r.s.s.'s. One is the r.s.s. of $\widehat{m}^*$ to $\widehat{m}$ under the condition $n=n_0$. Table \ref{Table:rss_est.pri_est_pre_cancer} shows the r.s.s. for several values of $n_0$. The other r.s.s. is that of $\widehat{m}$ to $\widehat{m}^p$ under the condition $(n, n^*)=(n_0, n_0^*)$. In Table \ref{Table:rss_est.pre_est_pool_cancer}, we show the r.s.s. for several cases where $n_0=n_0^*$. 
\begin{table}
\caption{Example 2. Breast cancer classification} 
\label{table:breast_cancer_prob}
\centering
\begin{tabular}{|c | c | c | c | c | c| }
\hline
   &  30-39  &  40-49  &  50-59  &  60-69  &  70-79 \\
\hline
1 &0.025 & 0.063 & 0.088 & 0.060 & 0.014 \\
\hline
2 &0.060 & 0.168 & 0.137 & 0.084 & 0.004 \\
\hline
3 &0.042 & 0.084 & 0.112 & 0.056 & 0.004 \\
\hline
Column sum & 0.126 & 0.317 & 0.337 & 0.200 & 0.021\\
\hline
\end{tabular}
\caption{Example 2. Risk of Estimators}
\label{Table:three_est_risk_cancer}
\centering
\begin{tabular}{l|l|c|c|c}
              &             & $n^*=200$ &  $n^*=600$ & $n^*=1000$ \\
\hline 
              & pre.app(pre.sim)   &  0.0367(0.0361)  &  0.0367(0.0361)  &  0.0367(0.0361) \\
$n=200$ &  pri.app(pri.sim)   &  0.0383(0.0384)  &  0.0315(0.0315)  &  0.0301(0.0302) \\
             &  pool.app(pool.sim) &  0.0320(0.0322)  &  0.0298(0.0302)  &  0.0290(0.0295) \\
\hline
              &  pre.app(pre.sim)    &  0.0119(0.0119)  &  0.0119(0.0119) &  0.0119(0.0119) \\
$n=600$ &   pri.app(pri.sim)   &  0.0188(0.0191)  &  0.0120(0.0122) &  0.0107(0.0108) \\
             &   pool.app(pool.sim) &  0.0110(0.0112)  &   0.0102(0.0104) &  0.0098(0.0100) \\
\hline
              & pre.app(pre.sim)    &  0.0071(0.0071)  &  0.0071(0.0071)  &  0.0071(0.0071) \\
$n=1000$&  pri.app(pri.sim)  &  0.0152(0.0154)  &  0.0085(0.0085)  &  0.0071(0.0072) \\
             &  pool.app(pool.sim) &  0.0067(0.0068)  &  0.0063(0.0064)   &  0.0061(0.0061) \\
\end{tabular}
\caption{Example 2. R.s.s. of $\widehat{m}^*$ to $\widehat{m}$}
\label{Table:rss_est.pri_est_pre_cancer}
\centering
\begin{tabular}{l|c|c|c|c|c}
$n_0$ & 200 & 400 & 600 & 800 & 1000 \\
\hline
$r.s.s.app(r.s.s.sim)$ & 236(255) & 433(445) & 633(652) & 832(850) & 1032(1029)\\
\end{tabular}
\caption{Example 2. R.s.s. of $\widehat{m}$ to $\widehat{m}^p$}
\label{Table:rss_est.pre_est_pool_cancer}
\centering
\begin{tabular}{l|c|c|c|c|c}
$n_0=n^*_0$ & 200 & 400 & 600 & 800 & 1000 \\
\hline
$r.s.s.app(r.s.s.sim)$ & 226(223) & 459(456) & 692(686) & 925(921) & 1158(1159)\\
\end{tabular}
\end{table}
We make the following observations:
\begin{itemize}
\item In the lower left part (including the diagonal part) of Table \ref{Table:three_est_risk_cancer},  $\widehat{m}$ is superior to  $\widehat{m}^*$, as is indicated in Section \ref{subsec: theo_result}. When $(n, n^*)=(200,600)$, $(200,1000)$, $(600, 1000)$, the result is opposite. From Table \ref{Table:rss_est.pri_est_pre_cancer}, we notice that when $n=200$, the prior survey needs 28\% more samples ($n^*=255$) than the present survey, while just 3\% more samples ($n^*=1029$) suffice when $n=1000$.
\item In all cases of Table \ref{Table:three_est_risk_cancer}, $\widehat{m}^p$ is superior to $\widehat{m}$. It seems that each $(n, n^*)$ in Table \ref{Table:three_est_risk_cancer} is large enough to guarantee the superiority of $\widehat{m}^p$ to $\widehat{m}$.  In view of the total sample size, the performance of  $\widehat{m}^p$ is inferior to $\widehat{m}$ as is shown in Section \ref{subsec: theo_result}. Table \ref{Table:rss_est.pre_est_pool_cancer} shows that  $\widehat{m}$ with the sample size of $1.12n_0-1.16n_0$ attain the same risk as $\widehat{m}^p$ with the total sample size $2n_0$. 
\item As the  theoretical result indicates, $\widehat{m}^p$ is always superior to $\widehat{m}^*$.  
\item There is no large discrepancy between the simulation and the approximation for the risk and the r.s.s. under the given sample sizes.
\end{itemize}
%
%
%
%
\bigskip
\textit{Example 3: Household classification}\\
 As the third example, we use data from ``2014 national survey of family income and expenditure'' by Statistics Bureau in Japan (https://www.stat.go.jp/english/data/zensho/
\\index.html). Table \ref{table:prob_income_age} is the relative frequency obtained from the cross tabulation of 100006 households by ``Yearly income group $(Y1,...,Y10)$'' and ``Household age group $(H1,...,H6)$''. We use this relative frequency as the population parameter $m_{ij},\ 1\leq i \leq 10,\ 1\leq j \leq 6$. 

The two-staged model is considered where the first-stage model is the full model over age groups $H1, ..., H6$, and each second-stage model is a full model over the income groups $Y1,..., Y10$ within an age group. We suppose that there is a prior survey on the first-stage model.  For the risk of the three estimators under several values of $(n,n^*)$, see Table \ref{Table:three_est_risk_household}. The r.s.s. of $\widehat{m}^*$ to $\widehat{m}$  under the condition $n=n_0$ is shown in Table \ref{Table:rss_est.pri_est_pre_household}. The r.s.s. of $\widehat{m}$ to $\widehat{m}^p$ under the condition $(n, n^*)=(n_0, n_0^*)$ is  presented in Table \ref{Table:rss_est.pre_est_pool_household}.  Because the results are  similar  to Example 2, we omit the comments.
\begin{table}
\caption{Classification of households by income and age}
\label{table:prob_income_age}
\centering
\footnotesize
\begin{tabular}{|c|c|c|c|c|c|c|}
\hline
    &  H1 & H2 & H3 & H4 & H5 & H6 \\ \hline
Y1 & 0.00161 & 0.00232 & 0.00512 & 0.00395 & 0.00468 & 0.00066  \\ \hline
Y2 & 0.00331 & 0.0081 & 0.00953 & 0.00783 & 0.01145 & 0.00278  \\ \hline
Y3 &0.00974 & 0.02109 & 0.02046 & 0.01499 & 0.02536 & 0.00494  \\ \hline
Y4 &0.00799 & 0.03519 & 0.03229 & 0.02017 & 0.0338 & 0.00708  \\ \hline
Y5 &0.00547 & 0.0376 & 0.04362 & 0.02442 & 0.02675 & 0.00398 \\ \hline
Y6 &0.00494 & 0.05082 & 0.09003 & 0.05772 & 0.03732 & 0.00452  \\ \hline
Y7 &0.00126 & 0.02106 & 0.0543 & 0.05531 & 0.01999 & 0.00234  \\ \hline
Y8 &0.00071 & 0.00961 & 0.0323 & 0.04043 & 0.0108 & 0.00122 \\ \hline
Y9 &0.00011 & 0.00201 & 0.01204 & 0.02184 & 0.00466 & 0.00052  \\ \hline
Y10 & 0.00006 & 0.00139 & 0.00697 & 0.01582 & 0.00344 & 0.00022 \\ \hline
Col. Sum & 0.03520 & 0.18919 & 0.30666 & 0.26248 & 0.17825 & 0.02826 \\ \hline
\end{tabular}
\caption{Example 3. Risk of Estimators}
\label{Table:three_est_risk_household}
\centering
\begin{tabular}{l|l|c|c|c}
              &             & $n^*=1000$ &  $n^*=2000$ & $n^*=3000$ \\
\hline 
              & pre.app(pre.sim)    &  0.0332(0.0300)  &  0.0332(0.0300)  &  0.0332(0.0300) \\
$n=1000$ & pri.app(pri.sim) &  0.0335(0.0304)  &  0.0323(0.0292)  &  0.0319(0.0287) \\
             &   pool.app(pool.sim)&  0.0321(0.0290)  &  0.0317(0.0286)  &  0.0315(0.0284) \\
\hline
              &  pre.app(pre.sim)   &  0.0157(0.0149)  &  0.0157(0.0149) &  0.0157(0.0149) \\
$n=2000$&   pri.app(pri.sim)   &  0.0170(0.0163)  &  0.0158(0.0150) &  0.0153(0.0146) \\
             & pool.app(pool.sim)   &  0.0153(0.0145)  &   0.0151(0.0143) &  0.0150(0.0142) \\
\hline
              &  pre.app(pre.sim)   &  0.0102(0.0099)  &  0.0102(0.0099)  &  0.0102(0.0099) \\
$n=3000$&   pri.app(pri.sim)   &  0.0120(0.0116)  &  0.0107(0.0104)  &  0.0103(0.0099) \\
             &   pool.app(pool.sim) &  0.0100(0.0097)  &  0.0099(0.0096)   &  0.0098(0.0095) \\
\end{tabular}
\caption{Example 3. R.s.s. of $\widehat{m}^*$ to $\widehat{m}$}
\label{Table:rss_est.pri_est_pre_household}
\centering
\begin{tabular}{l|c|c|c|c|c}
$n_0$ & 1000 & 1500 & 2000 & 2500 & 3000 \\
\hline
r.s.s.app(r.s.s.sim) & 1157(1183) & 1650(1668) & 2146(2145) & 2644(2696) & 3143(3155)\\
\end{tabular}
\caption{Example 3. R.s.s. of $\widehat{m}$ to $\widehat{m}^p$}
\label{Table:rss_est.pre_est_pool_household}
\centering
\begin{tabular}{l|c|c|c|c|c}
$n_0=n_0^*$ & 1000 & 1500 & 2000 & 2500 & 3000 \\
\hline
r.s.s.app(r.s.s.sim) & 1031(1035) & 1552(1562) & 2074(2079) & 2595(2604) & 3117(3129)\\
\end{tabular}
\end{table}
\section{Conclusion}
The golden rule ``the larger sample, the less risk'' holds true under sufficient sample sizes. Actually, we observed the following points theoretically, and numerically for three examples:
\begin{itemize}
\item When $n$ and $n^*$ are large enough, the risk of $\widehat{m}^p$ (which uses the sample of size $n+n^*$ for the estimation of the first-stage model)  is lower than that of $\widehat{m}$ (which uses the sample of size $n$ for the estimation of the first-stage model) or $\widehat{m}^*$ (which uses the sample of size $n^*$ for the estimation of the first-stage model).
\item When $n$ and $n^*$ are large enough and additionally $n^*>n$, then the risk of $\widehat{m}^*$ is lower than that  of $\widehat{m}$.
\end{itemize}
However, blending samples from two surveys has a negative effect in risk reduction.  We notice this from the following theoretical or numerical observations:
\begin{itemize}
\item When $n=n^*$, the risk of $\widehat{m}$ is lower than that of $\widehat{m}^*$. The latter is using two independent sample sets separately for the first and second-stage models, while the former uses the common sample for the two stages. Unified approach $\widehat{m}$ is preferable, because the estimation at the first stage gives an appropriate weight to the each result of the second-stage estimations.
\item If you compare $\widehat{m}$ using the present sample of size $n+n^*$ to $\widehat{m}^p$ (which uses the sample of a total size $n+n^*$), the former performs better with respect to the risk. In view of the r.s.s. of $\widehat{m}$ to $\widehat{m}^p$ under the condition $(n, n^*)=(n_0, n_0)$, it is far less than $2n_0$. 
\item If $n$ is small, we observe a rather pathological phenomenon that increasing $n^*$ causes the higher risk of $\widehat{m}^p$. 
\end{itemize}
The boundary between the area of $(n,n^*)$ where the golden rule holds, and the pathological area, can be estimated using the approximated risks. If the approximated risk of $\widehat{m}$ is smaller than that of $\widehat{m}^p$, that is, \eqref{difference_present_pool} is minus, then it indicates $n$ or $n^*$ is too small and we might be in the pathological area.  Consequently, we propose the following procedure.
\\
1. If the present survey is already finished, substitute $m_{i\cdot}$ in \eqref{difference_present_pool} with the estimate from $\widehat{m}^p_{i\cdot}$ and calculate the value. If the estimated value is plus (minus), use the estimator $\widehat{m}^p$ ($\widehat{m}$). \\
2. If the present survey is still in the planning stage, then substitute $m_{i\cdot}$ in  \eqref{difference_present_pool} with the estimate from $\widehat{m}^*_{i\cdot}$. With $n$ given, if the estimated value of \eqref{difference_present_pool} is minus, it indicates that the $n$ is not large enough for reliable estimation. We should try to increase $n$ (or alternatively decrease the number of the categories within each second-stage model) until the value gets positive.
\\
\\
\\
\\
%
%
%
%
%
\section*{Acknowledgment}
%
%
%
%

%
%
%
%
%
\section*{Appendix}
%
%
%
%
\begin{lemma}
\label{expectations_difference}
Let $X=(X_1,\ldots,X_{p+1})$ be the random vector whose distribution is defined as the multinomial distribution with \eqref{para_multi_dist} and the sample size $n$. The distribution under the condition $X_i \ne 0,  1\leq \forall i \leq p+1$ is considered.  Let the unconditional and conditional expectations of a random variable $Y(X)$ be denoted by $E[Y(X)]$ and $E^*[Y(X)]$, respectively. If 
$$
|Y(X)| \leq  a+b n^c 
$$
holds with some nonnegative numbers $a, b, c$, the difference between the two expectations decreases to zero with exponential speed as $n$ goes to infinity, namely for any $s>0$,
\begin{equation}
\label{diff_cond_uncond}
E^*[Y(X)]=E[Y(X)]+o(n^{-s}).
\end{equation}
In the special case, MLE of $m$, $\widehat{m}(X)$, and the moments of $X_i,\ i=1,\ldots,p+1$, the following  equations hold for any $s>0$.
\begin{align*}
E^*\bigl[\overset{-1}{D} [\widehat{m}(X) : m ]\bigr] &=E \bigl[\overset{-1}{D} [\widehat{m}(X) : m ]\bigr]+o(n^{-s}),\\
E^*[(X_i- n m_i)^k] &=  E[(X_i- n m_i)^k] +o(n^{-s}),\quad k=1,2,\ldots.
\end{align*}
\end{lemma}
\begin{proof}
For $i (=1,\ldots,p+1)$ and $s (>0)$, because the following equivalence relation holds
\begin{align*}
&n^s (1-m_i)^n \to 0 \\
&\iff s\log n +n \log(1-m_i) \to -\infty \iff n\Bigl(s\frac{\log n}{n}+\log(1-m_i)\Bigr) \to -\infty,
\end{align*}
we have
\begin{equation}
P(X_i=0)=(1-m_i)^n=o(n^{-s}).
\end{equation}
Let $Z_+^{p+1}$ be the set of all $p+1$-dimensional vectors whose elements are nonnegative integers.
\begin{align*}
\mathcal{X}&\triangleq \Bigl\{x=(x_1,\ldots,x_{p+1}) \in Z_+^{p+1} \Big| \sum_{j=1}^{p+1} x_j =n,\ x_i >0,\ 1\leq \forall i \leq p+1\Bigr\}\\
\mathcal{X}^*&\triangleq \Bigl\{x=(x_1,\ldots,x_{p+1}) \in Z_+^{p+1} \Big| \sum_{j=1}^{p+1} x_j =n, x_i=0,\ 1 \leq \exists i \leq p+1\Bigr\}
\end{align*}
Notice that
\[
m^*\triangleq P(X_i=0,\ 1\leq \exists i \leq p+1) \leq \sum_{i=1}^{p+1} P(X_i=0),
\]
which means $m^*=o(n^{-s})$. Because
\begin{align*}
E^*[Y(X)]&=\sum_{x \in \mathcal{X}} Y(x) P(X=x)/(1-m^*) \\
E[Y(X)]&=\sum_{x \in \mathcal{X}} Y(x) P(X=x) +\sum_{x \in \mathcal{X}^*} Y(x) P(X=x),
\end{align*}
we have
\[
E^*[Y(X)]-E[Y(X)]=\sum_{x \in \mathcal{X}} Y(x) P(X=x)\frac{m^*}{1-m^*}-\sum_{x \in \mathcal{X}^*} Y(x) P(X=x).
\]
Choose arbitrary $s(>0)$. Because $\mathcal{X}$ and $\mathcal{X}^*$ are both finite sets, for some nonnegative constants $a', a'', b', b''$, 
\begin{align*}
n^s\Bigl|\sum_{x \in \mathcal{X}^*} Y(x) P(X=x)\frac{m^*}{1-m^*} \Bigr|
&\leq n^s \sum_{x \in \mathcal{X}^*} (a+b n^c) \frac{m^*}{1-m^*}\\
&\leq a' \frac{n^s m^*}{1-m^*} + b' \frac{n^{s+c} m^*}{1-m^*}\\
n^s\Bigl|\sum_{x \in \mathcal{X}^*} Y(x) P(X=x) \Bigr|\leq n^s\sum_{x \in \mathcal{X}^*} (a+b n^c) m^*
&\leq a''n^s m^*+b'' (n^{s+c} m^*).
\end{align*}
Because of $m^*=o(n^{-t})$ for any $t(>0)$, 
\[
\frac{n^sm^*}{1-m^*}, \frac{n^{s+c} m^*}{1-m^*}, n^s m^*, n^{s+c} m^* \to 0
\]
as $n\to \infty$, which shows \eqref{diff_cond_uncond}. The rest is obvious from the fact 
$\overset{-1}{D}[\widehat{m}(X) : m ]$ is bounded (notice that $x\log x \to 0$ as $x \to 0$) and
\[
|(X_i - n m_i)^k| \leq \max (m_i^k, (1-m_i)^k) n^k,\quad k=1,2,\ldots.
\]
\end{proof}
%
%
%
%
%
\begin{lemma}
\label{converge->p_converge}
Suppose that $X_n,\ n=1,2,\ldots$ converges in probability to some positive constant as $n \to \infty$. 
If $F(x)=O(1)$ as $x \to \infty$, then $F(nX_n)=O_p(1)$. If $F(x)=o(1)$ as $x \to \infty$, then $F(nX_n)=o_p(1)$.
\end{lemma}
\begin{proof}
The first statement is proved as follows. If $F(x)=O(1)$, there exists some $M(>0)$ and $x_0$ such that
\begin{equation}
\label{bounded_F}
|F(x)| < M,\quad\forall x > x_0.
\end{equation}
Let $b >0 $ denote the point to which $X_n$ converges in probability. Choose an arbitrary $\epsilon (>0)$ and $\tau (b > \tau >0)$, then because of the convergence of $X_n$, we have $n_0$ such that
\[
P(b-\tau < X_n < b+\tau ) > 1-\epsilon, \quad \forall n > n_0,
\]
which means
\[
P(n(b-\tau) < nX_n) > 1-\epsilon, \quad \forall n > n_0.
\]
Let $N=\max\bigl(x_0/(b-\tau),n_0\bigr)$. Note that if $n > N$, then $n(b-\tau)>x_0$ and $n >n_0$. Therefore,  
\begin{equation}
\label{x_0 < nX_n}
P(x_0 < nX_n) > 1-\epsilon, \quad \forall n > N.
\end{equation}
Combining \eqref{bounded_F} and \eqref{x_0 < nX_n}, we have
\[
P(|F(nX_n)| < M) > 1-\epsilon, \quad  \forall n > N.
\]
Now we prove the second statement. Choose an arbitrary $\epsilon(>0)$. If $F(x)=o(1)$, there exists $x_0$ such that
\begin{equation}
\label{vanish_F}
|F(x)| < \epsilon,\quad\forall x > x_0. 
\end{equation}
Combining  \eqref{x_0 < nX_n} and \eqref{vanish_F}, we have
\[
P(|F(nX_n)| < \epsilon) > 1-\epsilon, \quad  \forall n > N.
\]
\end{proof}
\end{document}